\documentclass[12pt]{article}
\usepackage{amsmath,amsfonts,amsthm,amssymb,blindtext
,graphicx,pifont,fncylab,theoremref,sectsty,tocloft,titlesec,enumitem,
fmtcount,titletoc}
\date{}
\newtheorem{thm}{Theorem}
\newtheorem{lem}{Lemma}

\newtheorem{dfn}{Definition}

\newtheorem{rmk}{Remark}
\newtheorem{prp}{Proposition}
\newcommand{\C}{\Bbb C}

\begin{document}

\baselineskip 20pt

\title{A logarithmic characterization of Arakelian sets}

\author{Grigorios Fournodavlos,\footnote{Department of Mathematics \& Applied Mathematics, University of Crete, Voutes Campus, 70013 Heraklion, Greece, gfournodavlos@uoc.gr.} \footnote{Institute of Applied and Computational Mathematics,
FORTH, 70013 Heraklion, Greece.}\;\; Vassili Nestoridis,\footnote{Department of Mathematics, University of Athens, Panepistimioupolis, Athens 15784, Greece, vnestor@math.uoa.gr.}\;\; Spyros Pasias\footnote{Department of Mathematics \& Applied Mathematics, University of Crete, Voutes Campus, 70013 Heraklion, Greece, spyrospasias@uoc.gr.}}

\maketitle

\begin{abstract}
Arakelian's classical approximation theorem \cite{Ar} gives necessary and sufficient conditions such that functions can be uniformly approximated in (unbounded) closed sets $F\subset \mathbb{C}$ by entire functions. The conditions are purely topological and concern the connectedness of the complement of $F$. We give a new characterization of Arakelian sets in terms of logarithmic branches of functions $f\in A(F)$, which are continuous in $F$ and holomorphic in its interior $F^\circ$. Our proof is based on a contradiction argument and the counterexample function that we use is furnished by the Weierstrass factorization theorem. 
\end{abstract}

\section{Introduction}

Approximation theorems (Runge \cite{Run}, Mergelyan \cite{Mer}, Arakelian \cite{Ar}) play a central role in the field of complex approximation, since they give necessary and sufficient conditions for functions to be uniformly approximated on a given set by model functions, such as polynomials, rational functions or more generally entire/meromorphic functions. Our goal in the present paper is to give a new characterization of closed sets $F\subset \mathbb{C}$ in the complex plane, for which uniform approximation by entire functions is possible. 

To this end, define the classes
\begin{align*}
H(F)=&\,\lbrace  f:U\to\mathbb{C} \text{ holomorphic for some open set $U\supset F$}\rbrace,\\
 A(F)=&\,\lbrace  f:F\to\mathbb{C} \text{ holomorphic in }F^\circ\text{ and continuous in F}\rbrace,
\end{align*}
where $F^\circ$ is the interior of $F$.
\begin{dfn}\label{def:hole}
Any bounded connected component $B$ of $\mathbb{C}\setminus F$ is called a hole of $F$.
\end{dfn}
Arakelian's theorem \cite{Ar} gives necessary and sufficient conditions such that every function $f\in A(F)$ can be uniformly approximated on $F$ by entire functions..
\begin{dfn}[Arakelian set]\label{def:Arak}
Let $F\subset \mathbb{C}$ be a closed set. We say that $F$ is an Arakelian set, if and only if the following points hold:\\[5pt]
1. $F$ is without holes;\\[5pt]
2. The union of all holes of $F\cup K$ is bounded, for any compact subset $K\subset \mathbb{C}$. 
\end{dfn}
\begin{rmk}\label{rem:K}
Point 2. in Definition \ref{def:Arak} is equivalent to the one where $K$ is replaced by $\overline{D}_0(n)=\{z\in\mathbb{C}: |z|\leq n\}$, for any $n\in\mathbb{N}$. 
\end{rmk}
\begin{rmk}
In fact, Definition \ref{def:Arak} appears in a later work \cite{RR}, where a new proof of Arakelian's theorem was given (see also \cite{Got}). The original conditions in \cite{Ar} are:\\[5pt]
1. $\mathbb{C}\setminus F$ is connected;\\[5pt]
2. $\mathbb{C}\setminus F$ is locally connected at infinity.\\[5pt]
Nevertheless, they are easily seen to be equivalent.
\end{rmk}

\begin{dfn}[Uniform approximation set]
Let $F\subset \mathbb{C}$ be a closed set. We say that $F$ is a set of uniform approximation, if for every $\epsilon>0$ and $f\in A(F)$, there exists an entire function $g$ such that $|f(z)-g(z)|<\epsilon$, for all $z\in F$. 
\end{dfn}
\begin{thm}[Arakelian \cite{Ar}]\label{thm:Ar}
A closed set $F\subset \mathbb{C}$ is a set of uniform approximation if and only if $F$ is an Arakelian set. 
\end{thm}
\subsection{Main result}

We can now state the main result:
\begin{thm}\label{thm:main}
A closed set $F\subset \mathbb{C}$ is an Arakelian set, if and only if 
for  every $f\in A(F)$ with no zeros in $F$, there exists $g\in A(F)$ such that $e^g=f$ in $F$.
\end{thm}
At first glance, the existence of logarithmic branches of functions $f\in A(F)$ might seem to have little to do with the conditions in Definition \ref{def:Arak}. Intuitively, it is easier to make the connection by first stating yet another characterization of Arakelian sets, proved in \cite{F}, in terms of simply connected neighborhoods of $F$, that is to say, open sets $V\supset F$ whose connected components are simply connected.
\begin{thm}\label{thm:F}
A closed set $F\subset \mathbb{C}$ is an Arakelian set if and only if it possesses a neighborhood basis of simply connected open sets. 
\end{thm}
At least one of the directions in Theorem \ref{thm:main} is clearer now. If a function $f\in A(F)$ could be extended to a neighborhood of $F$, then the existence of an intermediate simply connected neighborhood would allow for a logarithmic branch of $f$ to be well-defined in the latter.

We should give credit here to the influential work of Gauthier-Pouryayevali \cite{GP}, which inspired us to pursue the above characterizations. In fact, the directions in both Theorems \ref{thm:main}, \ref{thm:F} where we assume that $F$ is an Arakelian set, are contained in \cite{GP}.

Our contribution in the present paper is to prove the reverse direction stated in Theorem \ref{thm:main}. We argue by contradiction, assuming $F$ is not an Arakelian set. Then a special function $f\in A(F)$ is constructed, which is finally proven not to satisfy the initial assumption. Interestingly, the function $f$ that we use is entire and it is produced by invoking the Weierstrass factorization theorem, unlike the meromorphic functions that are usually employed in such proofs \cite[Chapter IV, \S2.$C_2$]{Gaier}.

\subsection{Outlook}

Theorem \ref{thm:Ar} has been successfully extended to relatively closed subsets of planar domains $G\subset \mathbb{C}$ \cite{Ar2} and non-planar Riemann surfaces of finite genus \cite{GH,Sch}. One of the outstanding open problems in complex approximation is to characterize sets of uniform approximation in Riemann surfaces of infinite genus, where Arakelian sets (Definition \ref{def:Arak}) are not always sets of uniform approximation, see \cite{BG} for a counterexample. Moreover, it was shown by Scheinberg \cite{Sch} that such a characterization, if it exists, cannot be purely topological, but it has to take into account the complex analytic structure of the given Riemann surface. 

It would be interesting to seek potential extensions of Theorem \ref{thm:main} to Arakelian sets of planar domains $G\subset \mathbb{C}$, as in \cite{Ar2}. The present logarithmic characterization and our method of proof could be naturally extended to Arakelian sets of simply connected domains $G\subset\mathbb{C}$, see \cite[Corollary 2.15]{F} for one of the directions. However, we do not know how it could be generalized to planar domains $G$ with holes, even when $G$ is an annulus. Finally, coming up with a generalization of our characterization that would take into account the analytic structure of a given domain, could hopefully be of use for characterizing uniform approximation sets in Riemann surfaces.

\subsection{Acknowledgments}

We would like to thank Paul M. Gauthier for his interest in this work and insightful comments.
G.F. and S.P. would like to acknowledge the support of the ERC starting grant 101078061 SINGinGR, under the European Union's Horizon Europe program for research and innovation, and the H.F.R.I. grant 7126, under the 3rd call for H.F.R.I. research projects to support post-doctoral researchers.

\section{Basic lemmas}

{\it Notation:} In the rest of the paper, $D_z(\delta)=\{w\in \mathbb{C}: |w-z|<\delta\}$ and $\overline{D}_z(\delta)=\{w\in \mathbb{C}: |w-z|\leq\delta\}$.

The following lemma is a simple generalization of the fact that continuous logarithmic branches in open sets are holomorphic.
\begin{lem}\label{lem:Car}
Let $V\subset\mathbb{C}$ be an open set and $f:V\to\mathbb{C}$ a holomorphic function. If there exists a continuous function $g:V\to\mathbb{C}$, such that $e^g=f$ in $V$, then $g$ is holomorphic in $V$.
\end{lem}
\begin{proof}
Let $w\in V$ and let $\varepsilon>0$ such that $|f(w)|>\varepsilon$. Such an $\varepsilon>0$ exists, since $f(w)=e^{g(w)}\neq0$. We can then define a holomorphic branch of the logarithm $\log: D_{f(w)}(\varepsilon)\to\mathbb{C}$. By the openness of $V$ and the continuity of $f$, there exists $D_w(\delta)\subset V$, such that $f\big(D_w(\delta)\big)\subset D_{f(w)}(\varepsilon)$. Define $g_w: D_{w}(\delta)\to\mathbb{C}$, $g_w(z)=(\log\circ f)(z)$. By construction, $g_w$ is holomorphic in $D_{w}(\delta)$, as the composition of two holomorphic functions, and $e^{g_w(z)}=f(z)$, for all $z\in D_w(\delta)$. 
It follows that the ratio 
\begin{align*}
\frac{g(z)-g_w(z)}{2\pi i}\in\mathbb{Z},
\end{align*}
for all $z\in D_w(\delta)$. Since $g(z)$ is continuous and the previous ratio only takes integer values, it must be constant.
Hence, $g(z)$ is holomorphic in $D_w(\delta)$. Also, since $w\in V$ is arbitrary, we have that $g$ is holomorphic in $V$.
\end{proof}

The following lemma generalizes the previous one to compact sets, in the sense that if a logarithm branch is continuous in a compact set $K$, then it must belong to $H(K)$, while still serving as a logarithm branch of the same function.
\begin{lem}\label{lem:Nestor}
Let $K\subset\mathbb{C}$ be a compact set and $f\in H(K)$ with no zeros in $K$. Assume that there exists a continuous function $g:K\rightarrow\mathbb{C}$ such that $e^{g}=f$ in $K$. Then $g$ can be continuously extended in some open set $V\supset K$, such that $e^g=f$ in $V$ and $g$ is holomorphic in $V$. 
\end{lem}
\begin{proof}
By assumption we have that $e^{g(z)}=f(z)$, for all $z\in K$. It follows that there exists a continuous choice of angles $\phi(z)$ that defines a function $\phi:K\rightarrow \mathbb{R}$ such that $g(z)=\log |f(z)|+ i\phi(z)$. Here, of course, $\phi(z)\in arg(f(z))$  is pieced together from various branches of the argument function, forming a continuous function. Since $\phi$ is uniformly continuous in $K$, there exists $\eta>0$ such that for all $z_1,z_2\in K$ with $|z_1-z_2|<\eta$ we have $|\phi(z_1)- \phi(z_2)|<\frac{1}{10^3}$. 

Now for every $z\in K$, there exists $\delta_z\in(0,\frac{\eta}{2})$, such that $\phi(w)$ can be extended to $D_z(\delta_z)$, relative to the branch of $\phi(z)$, where each $D_z(\delta_z)$ is contained in the domain of holomorphy of $f$. Moreover, by shrinking each $\delta_z$ if necessary, we may assume that for every $z_1,z_2\in D_z(\delta_z)$ it holds $|\phi(z_1)-\phi(z_2)|<\frac{1}{10^3}$. 

{\it Claim:} Let $V=\bigcup_{z\in K}D_z(\delta_z)$. The collection of extensions of $\phi(z)$ in each $D_z(\delta_z)$ gives rise to a well-defined and continuous choice of angles of $f(z)$ in $V$. 

To show that the claim holds, we must show that $\phi$ is indeed single valued. To this end, suppose that we have a $z_0\in D_{z_1}(\delta_{z_1})\cap D_{z_2}(\delta_{z_2})$ for some $z_1$, $z_2\in K$. In order for $\phi(z_0)$ to be well-defined, 
we need to show that $\phi_1(z_0)=\phi_2(z_0)$, where $\phi_1$ and $\phi_2$ are the corresponding branches of the argument function associated with $\phi(z_1)$ and $\phi(z_2)$ respectively. Since $\phi_1(z_0)$ and $\phi_2(z_0)$ represent the same angle but possibly in a different branch, to show equality it suffices to show that $|\phi_1(z_0)-\phi_2(z_0)|<2\pi$. By the triangle inequality we have
\begin{align*}
|\phi_1(z_0)-\phi_2(z_0)|&=|\phi_1(z_0)-\phi(z_1)+\phi(z_1)-\phi(z_2)+\phi(z_2)-\phi_2(z_0)|\\
&<|\phi_1(z_0)-\phi(z_1)|+|\phi(z_1)-\phi(z_2)|+|\phi(z_2)-\phi_2(z_0)|\\
\tag{$|z_1-z_2|<\delta_{z_1}+\delta_{z_2}<\eta$}&<\frac{3}{10^3}<2\pi.
\end{align*}
This completes the proof of the claim.

By shrinking each $\delta_z>0$ if necessary, $z\in K$, we can also continuously extend the uniformly continuous function $\log |f(z)|$ in $V$, such that $g=\log|f(z)|+i\phi(z)$ is itself continuously extended in $V$ and well-defined. The conclusion follows by Lemma \ref{lem:Car}, since $e^g=f$ in $V$ by construction.
\end{proof}

The following lemma is standard. 
\begin{lem}\label{lem:wind.0}
Let $V$ be an open set and $f:V\rightarrow\mathbb{C}$ a non-vanishing holomorphic function. The following statements are equivalent:\\[5pt]
1. There exists a holomorphic function $g:V\rightarrow\mathbb{C}$ such that $e^{g}=f$ in $V$.\\[5pt]
2. For every closed rectifiable path $\gamma$ in $V$, $\int_\gamma\frac{f'(z)}{f(z)}dz=0$.
\end{lem}

We would like the outer boundary of a hole $B$ of a closed set $F$ to serve as a regular curve enclosing some special point $\zeta\in B\subset F^c$. Although the former boundary may be quite complicated, in the subsequent lemmas we show that finding such a curve is possible in the vicinity of $\partial B$.
\begin{dfn}\label{def:filling}
Let $U\subset\mathbb{C}$ be a bounded domain and let $B_i$, $i\in I$, denote the bounded connected components of its complement $U^c$. We call $V=U\cup\big(\cup_{i\in I}B_i\big)$ its filling and $\partial V$ the outer boundary of $U$.
\end{dfn}
\begin{lem}\label{lem:filling}
Let $U,V$ be as in Definition \ref{def:filling}. Then $V$ is a simply connected bounded domain and $\partial V\subset \partial U$. 
\end{lem}
\begin{proof}
Let $B_u$ denote the unbounded connected component of $U^c$. Note that $\mathbb{C}=U\cup\big(\cup_{i\in I}B_i\big)\cup B_u$ and $V=B_u^c$. 
Since $B_u$ is closed, $V$ must be open. 

If $G_1,G_2$ are two disjoint open sets whose union $G_1\cup G_2=V$, then $(G_1\cap U)\cup(G_2\cap U)=V\cap U=U$. Since $U$ is connected, either $G_1\cap U$ or $G_2\cap U$ is the empty set, say $G_1\cap U=\emptyset$. Hence,  $G_1$ must intersect some $B_{i_0}$. Since $B_{i_0}$ is connected and $(G_1\cap B_{i_0})\cup(G_2\cap B_{i_0})=V\cap B_{i_0}=B_{i_0}$, we infer that $G_2\cap B_{i_0}=\emptyset$ and hence, $B_{i_0}\subset G_1$. Let $z\in\partial B_{i_0}$ and let $\overline{D}_z(\delta)\subset G_1$. We observe that the union $B_{i_0}\cup \overline{D}_z(\delta)$ is connected, closed, and it strictly contains $B_{i_0}$. Hence, it cannot be solely contained in $U^c$. In other words, $\overline{D}_z(\delta)\cap U\neq\emptyset$, which is a contradiction since $G_1\cap U=\emptyset$. We conclude that $G_1,G_2$ as assumed above do not exist and therefore, $V$ is connected (ie. a domain).

Let $\mathbb{C}\cup\{\infty\}$ denote the Riemann sphere. Then $(\mathbb{C}\cup\{\infty\})\setminus V=B_u\cup\{\infty\}$, which is obviously connected since $B_u$ is connected and unbounded. Thus, $V$ is simply connected. 

Let $z\in\partial V$. Then for every $\delta>0$, $D_z(\delta)$ intersects $V$ and $V^c=B_u\subset U^c$. If $D_z(\delta)$ does not intersect $U$, then $B_u\cup \overline{D}_z(\frac{\delta}{2})\subset U^c$ is closed, connected, and it strictly contains $B_u$, since $D_z(\frac{\delta}{2})\cap V\neq \emptyset$. The latter cannot be, hence, $D_z(\delta)\cap U\neq \emptyset$ and $z\in \partial U$. Hence, $\partial V\subset \partial U$. This completes the proof of the lemma. 
\end{proof}
We say that a Jordan curve $\gamma$ encloses $\zeta$, if $\zeta$ is contained in the bounded connected component of $\mathbb{C}\setminus\{\gamma\}$. 
\begin{lem}\label{lem:Jordan}
Let $U,V$ be as in Definition \ref{def:filling} and let $\zeta\in U$. There exists an analytic Jordan curve $\gamma\subset V$, arbitrarily close to $\partial V$, such that $\gamma$ encloses $\zeta$.  
\end{lem}
\begin{proof}
Let $\delta=\text{dist}(\zeta,\partial V)>0$. 
Since $V$ is a simply connected and bounded domain, by the Riemann mapping theorem, there exists a biholomorphic function $h:V\to D_0(1)$. Let $0<\varepsilon<\delta$ and let $K_\varepsilon=\{z\in V: \text{dist}(z,\partial V)\ge \varepsilon\}$. Since $K_\varepsilon$ is compact, so is $h(K_\varepsilon)\subset D_0(1)$. Hence, there exists an $r\in(0,1)$, such that $h(K_\varepsilon)\subset D_0(r)$. Now let $\Gamma=\{(r\cos \theta, r\sin\theta): \theta\in[0,2\pi]\}$ and let $\gamma(\theta)=h^{-1}(\Gamma)$. We notice that $\gamma$ is an analytic Jordan curve contained in $V_\varepsilon=\{z\in V: \text{dist}(z,\partial V)<\varepsilon\}$. Moreover, $\mathbb{C}\setminus \{\gamma\}$ divides the plane in a bounded and an unbounded component. The bounded one is $B=h^{-1}\big(D_0(r)\big)\supset K_\varepsilon$, since it is open, connected, and its boundary coincides with $\gamma$. By construction, $\zeta\in K_\varepsilon\subset B$. Hence, $\gamma$ encloses $\zeta$.  
\end{proof}
\section{Proof of Theorem \ref{thm:main}}\label{sec:proof.thm}

The one direction of the characterization in Theorem \ref{thm:main} is contained in \cite{GP}. We include the proof here for the sake of completeness.
\begin{prp}[Gauthier-Pouryayevali \cite{GP}]\label{prop:GP}
Suppose that the closed set $F\subset \mathbb{C}$ is an Arakelian set. For every function $f\in A(F)$ with no zeros in $F$, there exists a function $g\in A(F)$ such that $e^g=f$ in $F$.
\end{prp}
\begin{proof}
By Tietze's extension theorem, there exists a continuous extension of $f$ in $\mathbb{C}$, which we
denote by $\widetilde{f}:\mathbb{C}\to\C$. Our assumption implies that the open set $U=\mathbb{C}\setminus \widetilde{f}^{-1}(\{0\})$ contains $F$.
Thus, by Theorem \ref{thm:F} there exists a simply connected open set $V$ with $F\subset V\subset U$.
If we consider the covering map $\exp: \mathbb{C}\to\mathbb{C}\setminus\{0\}$, then the latter implies that
$\widetilde{f}_{\big|V}:V\to\mathbb{C}\setminus\{0\}$ can be lifted
to a continuous function $\widetilde{g}:V\to\mathbb{C}$, such that $\widetilde{f}_{\big|V}=e^{\widetilde{g}}$.
The function $g=\widetilde{g}_{\big|F}$ is obviously continuous and $e^g=f$ in $F$.
Since $f_{\big|F^0}$ is holomorphic, by Lemma \ref{lem:Car}, $g$ is also holomorphic in $F^0$.
Thus, $g\in A(F)$ and the proof of the proposition is complete.
\end{proof}
The reverse direction is contained in the following proposition that, combined with Proposition \ref{prop:GP}, completes the proof of Theorem \ref{thm:main}. 
\begin{prp}\label{prop:FNP}
Let $F\subset\mathbb{C}$ be a closed set, such that for every $f\in A(F)$ with no zeros in $F$, there exists a $g\in A(F)$ satisfying $e^g=f$ in $F$. Then $F$ is an Arakelian set.
\end{prp}
\begin{proof}
{\it Step 1. $F$ has no holes.} Arguing by contradiction, suppose that $F$ has a hole $B$. Let $\zeta\in B$ and let $f(z)=z-\zeta$. Since $f$ has no zeros in $F$, there exists $g\in A(F)$, such that $e^g=f$ in $F$.  

Notice that $\partial B\subset F$ is compact. By Lemma \ref{lem:Nestor}, there exists an open neighborhood $U$ of $\partial B$, such that $g:U\to\mathbb{C}$ is extendible as a holomorphic function and satisfying $e^g=f$ in $U$. Consider $V$ to be the filling of $B$ and $\gamma\subset V$ an analytic Jordan curve, furnished by Lemma \ref{lem:Jordan}, enclosing $\zeta$ and being sufficiently close to $\partial V\subset \partial B$, such that $\gamma\subset U$. By Lemma \ref{lem:wind.0}, 
\begin{align*}
\frac{1}{2\pi i}\int_\gamma\frac{f'(z)}{f(z)}dz=0.
\end{align*}
On the other hand, by Cauchy's integral formula, the preceding LHS must be equal to 1, provided we give $\gamma$ a counter clockwise orientation. This completes the contradiction argument, which shows that $F$ has no holes.

{\it Step 2. The union of all holes of $F\cup \overline{D}_0(n)$ is bounded, for any $n\in\mathbb{N}$.} We argue again by contradiction. Suppose that there exists $n_0\in\mathbb{N}$, such that the union of all holes of $F\cup \overline{D}_0(n_0)$ is unbounded.\footnote{Note that if this is true, then the same will hold for all $n\ge n_0$, but it is not needed in the proof.} It follows that there exists a sequence of holes $B_i$ of $F\cup \overline{D}_0(n)$, $i\in\mathbb{N}$, and a sequence of points $\zeta_i\in B_i$, such that $\zeta_i\rightarrow\infty$. By the Weierstrass factorization theorem, there exists an entire function $f:\mathbb{C}\to\mathbb{C}$, vanishing exactly at the points $\zeta_i$ with multiplicity 1. Since $f(z)\neq 0$, for any $z\in F$, by assumption, there exists a function $g\in A(F)$, such that $e^g=f$ in $F$.

Let $S_{n_0}=\partial \overline{D}_0(n_0)$ and let $K=F\cap S_{n_0}$. Also, let $\widetilde{K}=F\cap \overline{D}_0(n_0+1)\supset K$. By Lemma \ref{lem:Nestor}, there exists a neighborhood $U\supset \widetilde{K}$ and an extension $\widetilde{g}:U\to\mathbb{C}$ of the restriction $g\big|_{\widetilde{K}}$, such that $\widetilde{g}$ is holomorphic and $e^{\widetilde{g}}=f$ in $U$. Let $2\delta=\min\{\text{dist}(K,U^c),1\}$ and consider the function 
\begin{align}\label{h}
h(z)=\left\{\begin{array}{ll}
g(z),& z\in F\\
\widetilde{g}(z),& z\in \overline{\cup_{w\in K}D_w(\delta)}
\end{array}\right.
\end{align}

{\it Claim:} $h(z)$ is continuous, well-defined, and satisfying $e^h=f$ in its domain of definition. 

{\it Proof of claim:}
Let $z\in F\cup\big(\overline{\cup_{w\in K}D_w(\delta)}\big)$. If $|z|>n_0+\frac{1}{2}$, then $z\in F\setminus \big(\overline{\cup_{w\in K}D_w(\delta)}\big)$ and $h(w)=g(w)$, for all $w$ sufficiently close to $z$. Hence, $h$ is single valued and continuous at $z$. On the other hand, if $|z|\leq n_0+\frac{1}{2}$, then for all $w$ sufficiently close to $z$, we have that 
\begin{align*}
h(w)=\left\{\begin{array}{ll}
g(w),& \text{if $w\in \widetilde{K}$}\\
\widetilde{g}(w),& \text{if $w\in \overline{\cup_{w\in K}D_w(\delta)}$}
\end{array}\right.=\widetilde{g}(w),
\end{align*}
since $w$ cannot belong in $F\setminus\widetilde{K}$. 
Hence, $h$ is single valued and continuous at $z$. The fact that $e^h=f$ in its domain of definition is obvious, since both $g,\widetilde{g}$ are logarithms of $f$ in their respective domains of definition. This proves our claim.


Notice that $\partial B_i\subset F\cup\overline{D}_0(n_0)$. Let $\Gamma_i=(\partial B_i\cap S_{n_0})\setminus F$. First, we argue that $\Gamma_i$ is open in the relative topology of the circle $S_{n_0}$. Indeed, if $z\in \Gamma_i\subset F^c$, then there exists $D_z(\varepsilon)\subset F^c$. Since $[D_z(\varepsilon)\setminus \overline{D}_0(n_0)]\cap B_i\neq\emptyset$ and $[D_z(\varepsilon)\setminus \overline{D}_0(n_0)]\cup B_i\subset [F\cup\overline{D}_0(n_0)]^c$ is connected, we conclude that $D_z(\varepsilon)\setminus \overline{D}_0(n_0)\subset B_i$. Hence, $D_z(\varepsilon)\cap S_{n_0}\subset \Gamma_i$. We thus have that $\Gamma_i$ is a countable (not necessarily finite) union of open arcs $A_{i,m}$, $m\in\mathbb{N}$. A similar argument also implies that $\Gamma_i\cap \Gamma_j=\emptyset$, for every $i\neq j$. Otherwise, there is a common arc $A\subset \Gamma_i\cap \Gamma_j\subset S_{n_0}$ and by choosing $z\in A$, $D_z(\varepsilon)$ appropriately, we conclude that $D_z(\varepsilon)\setminus \overline{D}_0(n_0) \subset B_i\cap B_j$, which cannot be since $B_i,B_j$ are disjoint. Lastly, the endpoints $z_{i,m}^\pm\in S_{n_0}$ of each arc $A_{i,m}$ must belong to $F$, otherwise we choose again a suitable disc $D_{z_{i,m}^\pm}(\varepsilon)$ and show that $A_{i,m}$ can be extended further from the given endpoint in the corresponding direction along $S_{n_0}$. 

To summarize the last paragraph, we have argued that the traces of the boundaries of all $B_i$'s on $S_{n_0}\setminus F$, consist of the pairwise disjoint sets $\Gamma_i$, each of which consists of disjoint open arcs $A_{i,m}$ of $S_{n_0}$, having endpoints $z_{i,m}^\pm\in F\cap S_{n_0}=K$. If $\lambda$ is the 1-dimensional Lebesgue measure on $S_{n_0}$, we then have that 
\begin{align}\label{Gamma.i.Leb}
\sum_{i=1}^{+\infty}\lambda (\Gamma_i)=\sum_{i,m=1}^{+\infty}\lambda (A_{i,m})\leq 2n_0\pi. 
\end{align}
In particular, there exists an $i_0$ such that $\lambda (\Gamma_{i_0})<\frac{\delta}{10^3}$. It follows that $A_{i_0,m}\subset D_{z_{i_0,m}^\pm}(\delta)$ for every $m$ and hence, $\Gamma_{i_0}\subset \overline{\cup_{w\in K}D_w(\delta)}$. We conclude that the whole boundary of $B_{i_0}$ is contained in $F\cup\big(\overline{\cup_{w\in K}D_w(\delta)}\big)$, where the extended logarithmic branch \eqref{h} of $f$ is well-defined.\footnote{In fact, all $\partial B_i$'s will eventually lie in the domain of $h$, but it is not needed in the proof.}

The contradiction argument is now completed by arguing for the most part as in {\it Step 1}. Employing Lemma \ref{lem:Nestor} once more, we further extend $h$ to a neighborhood $G$ of $\partial B_{i_0}$, such that $h: G\to\mathbb{C}$ is holomorphic and $e^h=f$ in $G$. If $V$ is the filling of $B_{i_0}$, by Lemma \ref{lem:Jordan}, there exists an analytic Jordan curve sufficiently close to $\partial V\subset \partial B_{i_0}$, such that $\gamma$ encloses $\zeta_{i_0}$ and $\gamma \subset G$. By Lemma \ref{lem:wind.0}, 
\begin{align*}
\frac{1}{2\pi i}\int_\gamma\frac{f'(z)}{f(z)}dz=0.
\end{align*}
However, by the argument principle, the previous LHS is equal to the number of zeros minus the number of poles of $f$, enclosed by $\gamma$, which is larger or equal to 1 in the present case (if we give $\gamma$ a counter clockwise orientation). This completes the proof of the proposition.
\end{proof}

\end{document}